\documentclass{ifacconf}

\usepackage{graphicx}      
\usepackage{natbib}        

\usepackage[utf8]{inputenc}
\usepackage[T1]{fontenc}

\usepackage{amsmath, amssymb, amsfonts, mathtools, tensor, mathrsfs}
\usepackage{units}
\usepackage{stmaryrd}
\usepackage{bm}
\usepackage{soul}    
\usepackage{xcolor}  

\usepackage{algorithm}
\usepackage{algpseudocode}

\usepackage{pgfplotstable}
\usepackage{graphicx}
\usepackage{pgfplots}
\pgfplotsset{compat=newest}
\usetikzlibrary{external}
\usepackage{tikz}
\usetikzlibrary{3d, calc, spy}
\tikzset{new spy style/.style={
  spy scope={
    magnification=2,
    size=2cm,
    connect spies,
    every spy on node/.style={rectangle, draw, thick},
    every spy in node/.style={draw, rectangle, fill=none}
  }
}}

\usepgfplotslibrary{groupplots}
\usepgfplotslibrary{colormaps}

\usepackage{multirow}
\usepackage{booktabs}
\usepackage{xcolor}
\definecolor{matlabblue}{rgb}{0, 0.4470, 0.7410}
\definecolor{matlabgreen}{rgb}{0.1330, 0.5450, 0.1330}
\definecolor{matlaborange}{rgb}{0.9290, 0.6940, 0.1250}
\definecolor{matlabpurple}{rgb}{0.4940, 0.1840, 0.5560}
\definecolor{matlabblack}{rgb}{0,0,0}

\usepackage[colorinlistoftodos]{todonotes}

\newtheorem{definition}{Definition}
\newtheorem{remark}{Remark}
\newtheorem{lemma}{Lemma}
\newtheorem{theorem}{Theorem}

\ifx\proof\undefined
  \newenvironment{proof}[1][Proof]{\par\noindent\textbf{#1. }}{\hfill$\square$\par}
\fi

\newcommand{\proj}{\text{proj}}

\newcommand{\TOC}[1]{{\color{black}#1}}

\newcommand{\B}[1]{{\color{black}#1}}

\allowdisplaybreaks
\begin{document}
\begin{frontmatter}

\title{Safe Adaptive Feedback Control via Barrier States\thanksref{footnoteinfo}} 

\thanks[footnoteinfo]{This research was supported in part by the Air Force Research Laboratory under the contract number FA8651-24-1-0019. Any opinions, findings, or recommendations in this article are those of the authors and do not necessarily reflect the views of the sponsoring agencies.}


\author[First]{Trivikram Satharasi}
\author[First]{Tochukwu E. Ogri}
\author[First]{Muzaffar Qureshi} 
\author[Second]{Kyle Volle}
\author[First]{Rushikesh Kamalapurkar}

\address[First]{University of Florida, Mechanical and Aerospace Engineering, Gainesville, FL 32608 USA \tt\footnotesize{(e-mail: t.satharasi@ufl.edu, tochukwu.ogri@ufl.edu, muzaffar.qureshi@ufl.edu, rkamalapurkar@ufl.edu.})}
\address[Second]{Torch Technologies, Shalimar, FL, USA (email: {\tt\footnotesize Kyle.Volle@torchtechnologies.com.})}

\begin{abstract}                
This paper presents a safe feedback control framework for nonlinear control-affine systems \TOC{with parametric uncertainty} by leveraging adaptive dynamic programming (ADP) with barrier-state augmentation. The developed ADP-based controller enforces control invariance by optimizing a value function that explicitly penalizes the barrier state, thereby embedding safety directly into the Bellman structure. The near-optimal control policy computed using model-based reinforcement learning is combined with a concurrent learning estimator to identify the unknown parameters and guarantee uniform convergence without requiring persistency of excitation. Using a barrier-state Lyapunov function, we establish boundedness of the barrier dynamics and prove closed-loop stability and safety. Numerical simulations on an optimal obstacle-avoidance problem validate the effectiveness of the developed approach.
\end{abstract}


\end{frontmatter}

\section{Introduction}
As autonomous systems become more common, safety is one of the considerations that affects their real-world deployment. For example, systems such as aircraft guidance and navigation, autonomous self-driving vehicles, and nuclear reactors have safety as one of their main design goals. Safety requirements have motivated research focused on understanding and designing controllers that guarantee safety under \TOC{uncertainty}. In control systems, the notion of safety is regarded as staying within the regions or states of favorable conditions and avoiding \TOC{unfavorable} regions and states. This notion was further formalized by \cite{SCC.Nagumo.ea1942, SCC.BLANCHINI.ea1999} \TOC{via} forward \TOC{invariance of sets with respect to} the flow of the dynamics.



Over the past decade, control barrier functions (CBFs) have become a popular tool for enforcing safety by guaranteeing forward invariance of prescribed sets, see \cite{SCC.Wieland.Allgoewer2007, SCC.Ames.Grizzle.ea2014}. 
While initial work on CBFs was model-based
recent research has focused on extending CBF-based feedback control to ensure safety under model uncertainty (\TOC{\cite{ SCC.Mahmud.Nivison.ea2021, SCC.Robey.Lindemann.ea2021, SCC.Lopez.Slotine.ea2020, SCC.Ogri.Qureshi.ea2025b}}). 

To guarantee safety, CBF-based reinforcement learning (RL) methods such as \cite{SCC.Cohen.Belta.ea2020, SCC.Deptula.Chen.ea2020, SCC.Marvi.Kiumarsi.ea2020} integrate model-free policy learning with model-based safety filtering by incorporating a CBF in the cost function of the optimal control problem, enforcing forward invariance of a given safe set during execution. However, the CBF-based RL approaches in \cite{SCC.Cohen.Belta.ea2020, SCC.Deptula.Chen.ea2020, SCC.Marvi.Kiumarsi.ea2020} can lead to a trade-off between safety and performance, with overly restrictive safety conditions hindering the ability of the system to achieve its stability objectives. To avoid this trade-off, \cite{SCC.Cohen.Belta.ea2023} and \cite{SCC.Cohen.Serlin.ea2023} propose a decoupled safety-aware controller by developing a dedicated safety controller which is paired with a \TOC{model-based reinforcement learning} (MBRL) controller, allowing the system to maintain safety without compromising stability (\cite{SCC.Kamalapurkar.Rosenfeld.ea2016}). However, the results in \cite{SCC.Cohen.Belta.ea2023} rely on an $L_{\infty}$ bound on the gradient of the Lyapunov-like barrier function to guarantee stability, which may be difficult to obtain near the boundary of the safe set. While \cite{SCC.Cohen.Serlin.ea2023} relaxes this bound by using zeroing barrier functions, the controllers in \cite{SCC.Cohen.Belta.ea2023} and \cite{SCC.Cohen.Serlin.ea2023} are performance-driven but not optimal. Furthermore, RL QP-based controllers, such as in \cite{SCC.Cohen.Serlin.ea2023,krstic2023inverse}, ensure safety by solving a quadratic program at each step, but this safety enforcement can modify the control input in ways that deviate from the original MBRL policy, potentially compromising the stability properties that the MBRL controller would otherwise guarantee. Recent developments have further shown that CBF–based quadratic programs can introduce undesirable asymptotically stable equilibria (\cite{Reis.Anguiar2021}), which may lead the closed-loop system to converge to an unintended trajectory or equilibrium point rather than the one prescribed by the learning-based controller.

\TOC{Another popular approach to safe control is the barrier transformation (BT) method (\cite{SCC.Yang.Vamvoudakis.ea2019, SCC.Greene.Deptula.ea2020, SCC.Mahmud.Nivison.ea2021}), which applies nonlinear coordinate transformations to system states. Unlike CBF-based techniques, the BT approach in \cite{SCC.Yang.Vamvoudakis.ea2019} reformulate the constrained state-feedback control problem into an equivalent unconstrained problem through a state transformation, enabling the direct design of controllers that inherently satisfy state constraints. 
Subsequent extensions of the method have relaxed restrictive assumptions such as persistence of excitation \cite{SCC.Greene.Deptula.ea2020} and have incorporated learning-based techniques to handle parametric uncertainty (\cite{SCC.Mahmud.Nivison.ea2021}). Although BT-based approaches yield verifiable safe state-feedback controllers, they are limited to box constraints.
}

To address these limitations, \cite{SCC.Almubarak.Sadegh.ea2025} introduced a barrier operator that constructs an intrinsic safety coordinate, termed the barrier state (BaS), which transforms the original system into an extended model in which safety is encoded as \B{boundedness of a part of} the state, yielding a forward-invariant safety cone in the augmented dynamics. This approach completely bypasses the requirement to synthesize a valid CBF, a key limitation of traditional CBFs. BaS is directly integrated into the system dynamics, ensuring that control decisions are made keeping current and the immediate future safety in consideration. \TOC{Unlike BT-based techniques, BaS-based safe control applies to a more general class of constraints.}

BaS appear directly in the system dynamics, enabling the design of safety-certified controllers via standard nonlinear control tools. A fundamental limitation of both BaS and CBF approaches is their dependence on complete knowledge of the system dynamics, which is rarely available in practice. While CBF extensions for uncertain systems (\cite{SCC.Taylor.Ames.ea2020, SCC.Lopez.Slotine.ea2020}) provide robustness, they often impose conservative control bounds that limit performance. Integrating BaS with parameter estimation (\cite{Aoun.Zhao, Sunni.Alumubarak}) uses adaptive techniques to maintain safety under uncertainty; however, approaches such as \cite{Sunni.Alumubarak} does not guarantee convergence of parameter estimates to their true values, leaving a gap between robust safety and accurate system identification.

\TOC{In this paper, we propose a safe feedback control framework for nonlinear affine systems with unknown parameters by combining integral concurrent learning (ICL), adaptive dynamic programming (ADP), and BaS augmentation. The main contributions of this work are as follows. First, by embedding the safety constraint into an augmented system state, the BaS-RL framework ensures that the controller maintains near-optimal safety even when the system parameters are unknown. 
Second, by integrating an actor-critic ADP framework along with an ICL-based parameter estimator, our approach learns the optimal value function and generates a control policy that balances safety, stability, and performance. Through a Lyapunov-based analysis, we show that the developed ICL-BaS-RL architecture simultaneously achieves the desired safety, stability, and performance objectives, marking a significant advancement over existing safe optimal control methods.
}
 
\section{Problem Formulation}	
Considers a nonlinear control-affine system of the form 
\begin{equation}
\label{eq:dynamics}
\dot{x} = Y(x)\theta + f(x) + g(x)u, \quad x(0) = x_{0}, 
\end{equation}
where $x \in \mathbb{R}^{n}$ is system state vector, $u \in \mathbb{R}^{m}$ is the control input, $\theta \in \mathbb{R}^{p}$ is a vector of unknown parameters, and $Y: \mathbb{R}^n \to \mathbb{R}^{n \times p} $, $f: \mathbb{R}^{n} \rightarrow \mathbb{R}^{n}$, and $g: \mathbb{R}^{n} \rightarrow \mathbb{R}^{n \times m}$ are the known system regressor matrix, drift dynamics, and control effectiveness matrix respectively.
\begin{assum}\label{ass:locallylipschitzfunctions}
The functions $f$,  $g$, and $Y$ are locally Lipschitz continuous and satisfy $f(0) = 0$, $Y(0) = 0$, and $0 < \|g(x)\| \leq \overline{G}$ for some $\overline{G} > 0$ and for all $x \in \mathbb{R}^{n}$.
\end{assum}
\begin{assum}\label{ass:thetacompactset}
\B{There exists a known $\overline{\theta} > 0$ such that $||\theta||\leq \overline{\theta}$.}
\end{assum}
\B{To facilitate the design, define} the projection operator
\begingroup\medmuskip=0mu\begin{equation}\label{eq:projectionOperator}
\operatorname{proj}_{\Theta}(\mu,v)
\coloneqq
\begin{cases}
v, & \text{if } \mu \in \operatorname{int}(\Theta),\\[4pt]
\left(I_p - \dfrac{\Gamma\mu\mu^\top}{\mu^\top \Gamma \mu}\right)v,
& \text{if } \mu \in \partial\Theta \text{ and } \mu^\top v > 0,\\[6pt]
v, & \text{otherwise\B{,}}
\end{cases}
\end{equation}\endgroup
where \B{$
    \Theta \coloneqq \big\{\, \theta \in \mathbb{R}^{p} \;\big|\; \|\theta\| \le \overline{\theta} \,\big\},
$} $I_p$ is the $p\times p$ identity matrix, $\Gamma\in\mathbb{R}^{p\times p}$ is a symmetric positive definite adaptation-gain matrix, and $\mu,v\in\mathbb{R}^p$.

The projection operator in \eqref{eq:projectionOperator} is the same as the projection operator introduced in \cite{SCC.Cai.Queiroz.ea2006} and in Appendix E of \cite{SCC.Krstic.Kanellakopoulos.ea1995}.


To define safety, let $\mathcal{S} \subset \mathbb{R}^{n}$ represent the zero super-level set of a continuously differentiable function $h \in \mathcal{C}^{1}(\mathbb{R}^{n};\mathbb{R})$, defined as
\begin{align}
    \mathcal{S} &= \{x \in \mathbb{R}^{n} \mid h(x)\geq 0\}, \label{eq:safeSet1}\\
    \partial\mathcal{S} &= \{x \in \mathbb{R}^{n} \mid h(x) = 0\}, \\
    \operatorname{int}(\mathcal{S}) &= \{x \in \mathbb{R}^{n} \mid h(x) > 0\}, \label{eq:safeSet3}
\end{align}
where $\partial\mathcal{S}$ and $\operatorname{int}(\mathcal{S})$ denote the boundary and interior of $\mathcal{S}$, respectively. 
\begin{assum}

\B{The set $\mathcal{S}$ is nonempty, contains no isolated points, and $\nabla h(x) \neq 0, \forall x \in \partial\mathcal{S} $. }
\end{assum}

Let $t \mapsto x(t)$ denote the trajectory of \eqref{eq:dynamics} starting from $x_{0}$ under the input $u(\cdot)$. The notion of safety adopted in this work is based on conditional control invariance as formalized by the following definition.

\begin{definition}\label{defn:safety}
A set $\mathcal{S}\subset\mathbb{R}^{n}$ is said to be conditionally control invariant \B{with respect to $\mathcal{S}_0\subseteq \mathcal{S}$} if there exist\B{s} an admissible control policy $u(\cdot) \in\mathcal{U}$ such that, for every initial condition $x_{0}\in \mathcal{S}_{0}$, the corresponding trajectory satisfies $x(t) \in \mathcal{S}, \, \forall\, t\ge 0$.
\end{definition}



The objective of this paper is to estimate the unknown parameters of the system \eqref{eq:dynamics} while simultaneously synthesizing a safe feedback control policy $\pi$ that ensures \TOC{conditional} control invariance as described in Definition~\ref{defn:safety}.

\section{Safety-Embedded Model}

To integrate constraint-awareness directly into the system model in a manner that preserves differentiability, supports observer and estimator design, and admits Lyapunov-based guarantees, consider a barrier operator \B{$B \in \mathcal{C}^{\infty}\left((0, \infty); \mathbb{R}\right)$ such that} $\frac{dB(a)}{da} < 0 \;\; \forall a>0$,  $\lim_{a\to 0^+} B(a) = \infty$,
\B{and} $\frac{dB(a)}{da} \circ B^{-1} \in \mathcal{C}^{\infty}\left((0, \infty); \mathbb{R}\right)$.  
For \B{the} continuously differentiable function $h \in \mathcal{C}^{1}(\mathbb{R}^{n};\mathbb{R})$ \B{that defines $\mathcal{S}$}, define the composed barrier function
\begin{equation}\label{eq:barrierFunction}
\beta(x) \coloneqq B\bigl(h(x)\bigr).
\end{equation}  
Because $B$ diverges as $h(x)\to 0^+$,  $\beta$ satisfies
\begin{equation}
\beta(x) \ge 0, \;\; \forall x \in \operatorname{int}(\mathcal{S}), \qquad
\lim_{x \to \partial \mathcal{S}} \beta(x) = \infty,
\end{equation}
so the safety condition $h(x)>0$ becomes equivalent to $\beta(x)<\infty$. The following lemma is classical; however, we repeat it for completeness.

\begin{lemma}\label{lem:safety}(\cite{SCC.Almubarak.Sadegh.ea2025})
Let $\mathcal{S}\subset\mathbb{R}^{n}$ be the safe set defined in~\eqref{eq:safeSet1}, and let 
$\beta:\mathbb{R}^{n}\to[0,\infty]$ be the \B{composed} barrier function associated with $\mathcal{S}$ as specified in~\eqref{eq:barrierFunction}.  \TOC{Then for every $x_0\in\mathcal S_0$, the admissible control policy $u(\cdot) \in \mathcal{U}$ is safe, i.e.,} $\mathcal{S}$ is conditionally control invariant with respect to $\mathcal{S}_{0}\subseteq\mathcal{S}$, \TOC{if and only if} $\beta(x)<\infty$  \TOC{$\forall x\in\mathcal{S}$}.
\end{lemma}

Lemma~\ref{lem:safety} motivates the safe control strategy to embed safety into system dynamics, since as long as $\beta(x(t))$ remains finite, the state remains safely within~$\mathcal{S}$. 

\B{U}sing a coordinate transformation\TOC{,} the BaS \TOC{is defined as} $z \coloneqq \beta(x) - \beta_0$ with $\beta_0 \coloneqq \beta(0)$. The dynamics in these coordinates become
\begin{equation}\label{eq:ZDynamics}
\dot{z} = \Phi(z + \beta_0) \, \nabla h(x) \Big( Y(x)\theta + f(x) + g(x) u \Big).
\end{equation}
\B{where $\Phi = \frac{B(h)}{dh} \circ B^{-1} \in \mathcal{C}^{\infty}\left((0, \infty); \mathbb{R}\right)$ appears as a nonlinear gain induced by the barrier operator, and the initial BaS is $z_0= \beta(x_0)-\beta_0$.}

The BaS induces a forward-invariant cone in the extended $(x,z)$-space that encodes safety as boundedness of $z$.
\begin{lemma}
\label{lem:cone}
\TOC{Define} the set $\mathcal{C} \coloneqq \{(x,z)\in\mathbb{R}^{n}\times\mathbb{R} \mid z = \beta(x)-\beta_0\}$.
\TOC{Consider any solution $(x(t),z(t))$ of \eqref{eq:dynamics} and \eqref{eq:ZDynamics} under a locally Lipschitz feedback $u(x)$. Then $z(t) = \beta(x(t)) - \beta_0$ for all $t \ge 0$ if and only if $(x(0), z(0)) \in \mathcal{C}$. Moreover, for any $t \ge 0$, $x(t) \in \mathcal{S}$ if and only if $\beta(x(t)) < \infty$ if and only if $z(t) < \infty$}.
\end{lemma}
\begin{proof}
\TOC{Differentiate $\beta(x(t)) - \beta_0$ along solutions and compare with $\dot z$ in \eqref{eq:ZDynamics}. By uniqueness of solutions under local Lipschitzness, the two scalar trajectories coincide whenever they agree at $t=0$. The final equivalences follow from the first part and the defining property $\beta(x) = \infty$ if and only if $x \notin \mathcal{S}$}.
\end{proof}

\begin{remark}
The set $\mathcal{C}\subset\mathbb{R}^{n}\times\mathbb{R}$ is the graph of the mapping 
$x \mapsto \beta(x)-\beta_{0}$.  
While $\mathcal{C}$ is not a subset of $\mathcal{S}$, its projection onto the state space satisfies
$\pi_{x}(\mathcal{C})=\{x\in\mathbb{R}^{n}\mid \beta(x)<\infty\}=\B{\text{int}}(\mathcal{S})$.
Thus, $\mathcal{C}$ provides a lifted representation of the interior of the safe set in the augmented $(x,z)$-coordinates.
\end{remark}


Lemma~\ref{lem:cone} formalizes that the safety condition is dynamically equivalent to forward-boundedness of a lifted coordinate, allowing the use of Lyapunov arguments directly on $z$ rather than on $\beta(x)$.

Since $\theta$ in \eqref{eq:ZDynamics} is unknown, we introduce the estimated \TOC{BaS}, denoted \TOC{by} $\hat{z} \in \mathbb{R}$, which evolves according to
\begin{multline}\label{eq:ZhatDynamics}
\dot{\hat{z}} = \Phi(z + \beta_0) \, \nabla h(x) \Big( Y(x)\hat{\theta} + f(x) + g(x) u \Big) + \gamma( z - \hat{z} ),
\end{multline}
where $\gamma > 0$ is a user-defined observer gain.  

Let the \B{BaS} estimation error be defined as $\tilde{z} \coloneqq z - \hat{z}$. Subtracting \eqref{eq:ZhatDynamics} from \eqref{eq:ZDynamics} yields the error dynamics
\begin{equation}\label{eq:ZtildeDynamics}
\dot{\tilde{z}} = \Phi(z + \beta_0) \, \nabla h(x) Y(x) \tilde{\theta} - \gamma \tilde{z},
\end{equation}
where $\tilde{\theta} \coloneqq \theta - \hat{\theta}$ is the parameter estimation error.

The next section introduces a parameter-estimation law tailored to \eqref{eq:ZtildeDynamics} that preserves the forward-invariant cone of Lemma~\ref{lem:cone} while ensuring stability and boundedness of $(\tilde{z},\tilde{\theta})$ under relaxed excitation conditions.

\section{Parameter Estimator Design}
\label{section:parameterEstimator}

In this section, an ICL update law \B{(\cite{Chowdhary2010})} is developed to estimate the unknown parameters. The parameter estimator relies on the fact that the difference \B{between the estimated and the true BaS} at time $t$ and time $t-T$ can be expressed as an affine function of the parameters $\theta$ and a residual that reduces with reducing \B{BaS} estimation error.

\begin{lem}\label{lem:ErrorTermformulation}
For any fixed delay $T>0$ and all $t\ge T$, the system dynamics satisfy the incremental relation
\begin{equation}\label{eq:incrementalRelation}
    \mathcal{X}(t)
    \;=\;
    \mathcal{Y}(t)\theta + \mathcal{G}_{fu}(t),
\end{equation}
where $\mathcal{X}(t) \coloneqq x(t) - x(t - T)$, $\mathcal{Y}(t) \coloneqq \int_{t-T}^{t} Y(s(\tau))\,d\tau$, and $\mathcal{G}_{fu}(t) \coloneqq \int_{t-T}^{t} \big(f(s(\tau)) + g(s(\tau))u(\tau)\big)\,d\tau$.
\end{lem}
\begin{proof}
The result follows directly from the Fundamental Theorem of Calculus applied to \eqref{eq:dynamics}.
\end{proof}

\subsection{ICL Update Law and Data Management}
Lemma~\ref{lem:ErrorTermformulation} implies that, for any time instant $\tau$, the parametric regression error satisfies
\begin{equation}\label{eq:parametricError}
    \mathcal{Y}(t)\,\tilde{\theta}(\tau)
    \;=\;
    \mathcal{X}(t) - \mathcal{G}_{fu}(t) - \mathcal{Y}(t)\hat{\theta}(\tau),
\end{equation}
which motivates the following ICL update law
\begin{equation}
\dot{\hat{\theta}} = \proj_{\Theta}\big(\hat{\theta},\, \Gamma\phi\big), \label{eq:thetaUpdate}
\end{equation}
\TOC{where $\Gamma \in \mathbb{R}^{p \times p}$ is a time-varying positive definite least-squares gain matrix that satisfies
\begingroup\medmuskip=0mu\begin{equation}\label{eq:gammaUpdate}
   \scalebox{0.93}{$ \dot{\Gamma} = \begin{cases} 
       \displaystyle \beta_{\theta}\Gamma - k_{\theta}\Gamma \sum_{i=1}^{N} \frac{\mathcal{Y}_i^{\top}\mathcal{Y}_i}{1+\kappa\|\mathcal{Y}_i\|^{2}} \Gamma, & \text{if } \hat{\theta} \in \operatorname{int}(\Theta)\\[-5pt]
             &  \text{or } ( \hat{\theta} \in \partial\Theta
        \text{ and } (\Gamma\phi)^{\top}\hat{\theta} \le 0 ),\\
       0 & \text{otherwise},
   \end{cases} $} 
\end{equation}\endgroup
and} \begin{multline}
    \phi\coloneqq Y(x)^{\top}\nabla h(x)^{\top}\Phi(z+\beta_0)^{\top} \tilde{z}
\\+ k_{\theta}\sum_{i=1}^{N}
\frac{\mathcal{Y}_i^{\top}\big(\mathcal{X}_i - \mathcal{G}_{fu,i} - \mathcal{Y}_i\hat{\theta}\big)}{1+\kappa\|\mathcal{Y}_i\|^{2}},
\end{multline} with $\mathcal{X}_i \coloneqq \mathcal{X}(t_i)$, $\mathcal{Y}_i \coloneqq \mathcal{Y}(t_i)$, $\mathcal{G}_{fu,i}\coloneqq \mathcal{G}_{fu}(t_i)$. 
The gains $k_{\theta},\beta_{\theta},\kappa\in\mathbb{R}_{>0}$ are user-defined adaptation gains. Since $\Gamma(t)$ is positive definite, there exist constants $0<\underline{\Gamma}\le \overline{\Gamma}$ such that, for all $t\ge 0$,
\begin{equation}\label{eq:GammaBound}
\underline{\Gamma}\, I_{p} \;\preceq\; \Gamma(t) \;\preceq\; \overline{\Gamma}\, I_{p}.
\end{equation}

To facilitate data reuse, a finite set of delayed samples is maintained in a \emph{history stack}
\begin{equation}\label{eq:historyStackDef}
    \mathcal{H}
    \coloneqq
    \big\{\,\mathcal{X}_i,\,\mathcal{Y}_i,\,\mathcal{G}_{fu,i}\,\big\}_{i=1}^{N},
\end{equation}
which is updated at discrete times $t_i$ based on data persistence conditions.


To facilitate the analysis, let the parameter estimation error be defined as $\tilde{\theta} \coloneqq \hat{\theta} - \theta$. The corresponding dynamics follow from \eqref{eq:thetaUpdate} as
\begingroup\begin{equation}\label{eq:tildeThetaDyn}
\scalebox{0.95}{$\dot{\tilde{\theta}} 
= \proj_{\Theta}\left(\hat{\theta}, 
\left(- k_{\theta} \Gamma
\Sigma_{\mathcal{Y}}\tilde{\theta} 
-\Gamma Y(x)^{\top}\nabla h(x)^{\top}\Phi(z+\beta_0)^{\top} \tilde{z}\right) 
\right)$},
\end{equation}\endgroup
where $\Sigma_{\mathcal{Y}} \in \mathbb{R}^{p\times p}$ denotes the weighted regressor matrix defined as
$\Sigma_{\mathcal{Y}} \coloneqq \sum_{i=1}^{N} \sigma_i \mathcal{Y}_i^\top \mathcal{Y}_i$ with 
$\sigma_i \;=\; \frac{1}{1 + \kappa \|\mathcal{Y}_i\|^2} > 0$. 


\begin{assum}\label{ass:fullRank}
There exists a finite collection of sampling instants $\{t_i\}_{i=1}^{N}$ such that
\begin{equation}\label{eq:rankCond}
    \lambda_{\min}(\Sigma_{\mathcal{Y}})
    \eqqcolon \underline{\sigma}_{\theta} > 0.
\end{equation}
\end{assum}
\begin{remark}
The assumption is commonly adopted in ICL-based parameter estimation frameworks (\cite{Chowdhary2010, SCC.Parikh.Kamalapurkar.ea2019,  SCC.Ogri.Bell.ea2023, SCC.Ogri.Qureshi.ea2025b}) and entails that the system undergo sufficient excitation. Unlike the classical PE condition, however, it only demands excitation over a finite interval, rendering it a weaker requirement.
\end{remark}

A history stack satisfying \eqref{eq:rankCond} is said to be \emph{full rank}.  The constant $\underline{\sigma}_{\theta}$ quantifies the informational richness of the stored data and directly influences the rate of convergence of~\eqref{eq:thetaUpdate}.

To guarantee sufficient excitation and improve estimator conditioning, the history stack $\mathcal{H} = \{(\mathcal{X}_i, \mathcal{Y}_i, \mathcal{G}_{fu,i})\}_{i=1}^{N}$ is updated using a minimum-eigenvalue maximization strategy (\cite{Chowdhary2010}). A candidate data point $(\mathcal{X}^{*}, \mathcal{Y}^{*}, \mathcal{G}_{fu}^{*})$ is admitted by replacing an existing entry $(\mathcal{X}_j, \mathcal{Y}_j, \mathcal{G}_{fu,j})$ if it strictly increases the smallest eigenvalue of \B{$\Sigma_{\mathcal{Y}}$}, i.e.,
\begin{equation}\label{eq:eigMaxRuleRewritten}
\lambda_{\min}\Big(
\Sigma_{\mathcal{Y}}^{o}
\Big)
\;<\;
\frac{
\lambda_{\min}\Big(
\Sigma_{\mathcal{Y}}^{*}
\Big)
}{1+\delta}, \qquad \delta > 0,
\end{equation}
where $\Sigma_{\mathcal{Y}}^{o} = \sum_{i \neq j} \sigma_i \mathcal{Y}_i^\top \mathcal{Y}_i + \sigma_j \mathcal{Y}_j^\top \mathcal{Y}_j$,
$\Sigma_{\mathcal{Y}}^{*} = \sum_{i \neq j} \sigma_i \mathcal{Y}_i^\top \mathcal{Y}_i + \sigma^{*} \mathcal{Y}^{*\top} \mathcal{Y}^{*}$, and $\sigma^{*} = \frac{1}{1 + \kappa \|\mathcal{Y}^{*}\|^2}$.
By construction, \eqref{eq:eigMaxRuleRewritten} ensures that the minimum eigenvalue of $\Sigma_{\mathcal{Y}}$ is non-decreasing over time. This provides a rigorous guarantee that the history stack maintains \emph{numerical full rank} \TOC{once reached} and improves the conditioning of the concurrent learning update law (\cite{Chowdhary2010})\TOC{.}

In the following section, we develop a stabilizing control law for the augmented system in \eqref{eq:augmentedSystem} that minimizes the cost functional introduced in \eqref{eq:costFunctional} and satisfies the safety requirements of the original system \eqref{eq:dynamics}.

\section{Control Design}
\B{Lemma \ref{lem:safety} and \ref{lem:cone} indicate that the trajectories that remain in the safe set yield bounded BaS. Hence, to design a control policy that enforces the boundedness of BaS, we augment it to the state space of the system.}

To construct the safety-embedded system, we define the augmented state $s \coloneqq [x^\top, z]^\top \in \mathbb{R}^{n + 1}$. The corresponding augmented dynamics take the form
\begin{equation}\label{eq:augmentedSystem}
\dot{s} = A(s)\theta + F(s) + G(s) u, \quad s(0) = s_{0},
\end{equation}
where $A(s) \coloneqq [
    Y(x)^{\top}, \Phi(z + \beta_{0})\nabla h(x) Y(x)]^{\top}$, $F(s) = [f(x)^{\top}, \Phi(z + \beta_{0})\nabla h(x) f(x)]^{\top}$, 
and $G(s) = [ g(x)^{\top}, \Phi(z + \beta_{0}) \, \nabla h(x) g(x)]^{\top}$. The construction of $z$ ensures that conditional control invariance of the safety set $\mathcal{S}$ is equivalent to \B{uniform local} boundedness of the $(n+1)^{\rm{st}}$ coordinate of $s$, so the constraint is now encoded as a property of the dynamics rather than of the state space.

\subsection{Optimal Control Problem}
Let $Q\in\mathbb{S}_{+}^{n+1}$  and $R\in\mathbb{S}_{+}^{m}$, where $\mathbb{S}_{+}^{n}$ is a positive definite square matrix of order $n$.  
For a measurable input $ {u} :\mathbb{R}_{\ge 0}\to\mathbb{R}^{m}$, define the cost functional
\begin{equation}\label{eq:costFunctional}
    J(s_{0}, u) = \int_{0}^{\infty} s(\tau)^{\top}Qs(\tau) + u(\tau)^{\top}Ru(\tau)\, d\tau,
\end{equation}
where $s(\cdot)$ is the solution of \eqref{eq:augmentedSystem} \TOC{starting from initial state $s_{0}$ and under control policy $u(\cdot)$}.  
\begin{remark}
\TOC{In the following, we assume that an solution of the optimal control problem exists. In doing so, we implicitly assume that the augmented system \eqref{eq:augmentedSystem} is stabilizable. If the system is not stabilizable, then the dynamic compensation technique proposed in \cite{SCC.Almubarak.Sadegh.ea2025} can be employed to ensure stabilizability.}
\end{remark}
The value function \B{of the optimal control problem to minimize $J$ is defined as}
\begin{equation}\label{eq:valueFunction}
    V^{*}(s_{0})
    \coloneqq
    \inf_{u(\cdot)}
    J(s_{0},u),
\end{equation}
with the infimum taken over all admissible inputs for which \eqref{eq:costFunctional} is well-defined.  
\begin{assum}\label{eq:ValueDifferentiability}
The optimal value function $V^{*} : \mathbb{R}^{n+1} \to \mathbb{R}$ is continuously differentiable on $\mathbb{R}^{n+1}$, i.e.,
$V^{*} \in C^{1}(\mathbb{R}^{n+1}; \mathbb{R})$. 
\end{assum}
Under Assumption~\ref{eq:ValueDifferentiability}, the gradient $\nabla V^{*}(s)$ exists and is continuous for all $s \in \mathbb{R}^{n+1}$. Hence, $V^{*}$ satisfies the stationary Hamilton--Jacobi--Bellman (HJB) equation \B{\cite[Theorem 1.5]{SCC.Kamalapurkar.Walters.ea2018}}.
\begin{equation}\label{eq:HJB}
    \inf_{u\in\mathbb{R}^{m}}
    \Big\{
        \nabla V^{*}(s)^{\top}\big(F(s)+G(s)u\big)
        + s^{\top} Q\, s
        + u^{\top} R\, u
    \Big\}
    = 0,
\end{equation}
\TOC{for all $s \in \mathbb{R}^{n+1}$}. The Hamiltonian associated with \eqref{eq:HJB} is strictly convex in $u$ due to $R\succ 0$.  
Thus, for each $s\in\mathbb{R}^{n+1}$, the minimizing control is uniquely determined by the first-order optimality condition
\begin{multline}
    0
    =
    \frac{\partial}{\partial u}
    \Big(
        \nabla V^{*}(s)^{\top}G(s)u 
        + u^{\top}Ru
    \Big)
    \\ =
    G(s)^{\top} \nabla V^{*}(s) + 2Ru,
\end{multline}
\TOC{which} yields the optimal feedback control policy
\begin{equation}\label{eq:optimalU}
    u^{*}(s)
    =
    -\frac{1}{2}
    R^{-1} G(s)^{\top} \nabla V^{*}(s).
\end{equation}
Substitution of \eqref{eq:optimalU} into \eqref{eq:HJB} yields the reduced-form HJB equation
\begin{multline}\label{eq:HJBReduced}
    \nabla V^{*}(s)^{\top}F(s)
    + s^{\top} Q\, s
    \\ - \frac{1}{4}
        \nabla V^{*}(s)^{\top}
        G(s) R^{-1} G(s)^{\top}
        \nabla V^{*}(s)
    = 0.
\end{multline}
Equations \eqref{eq:optimalU} characterize the optimal controller for the safety-embedded dynamics.  
Notably, the BaS coordinate transforms the state-\B{constrained optimal control} problem into an unconstrained \TOC{optimal control} problem, with safety encoded implicitly through the structure of $F$ and $G$.

\subsection{Value Function Approximation}

For general nonlinear control systems, the HJB equation associated with the optimal control problem \eqref{eq:augmentedSystem}--\eqref{eq:HJBReduced} rarely admits a closed-form solution. To address this, we adopt an approximate dynamic programming (ADP) formulation over a compact set $\Omega \subset \mathbb{R}^{n+1}$ with \B{ $\Omega \supset \mathcal{C}$}, following the approximation architecture proposed in~\cite{SCC.Bhasin.Kamalapurkar.ea2012}.  

Let $L \in \mathbb{N}$ and let $\sigma : \mathbb{R}^{n+1} \to \mathbb{R}^{L}$ denote a vector of continuously differentiable basis functions, $\sigma \in C^{1}(\mathbb{R}^{n+1};\mathbb{R}^{L})$, and let $W \in \mathbb{R}^{L}$ be an unknown parameter vector. The optimal value function $V^{*}$ is said to admit a {parametric approximation on $\Omega$ if there exists a continuously differentiable residual function $\epsilon : \Omega \to \mathbb{R}$ such that, for all $s \in \Omega$,
\begin{equation}
    V^{*}(s) = W^{\top}\sigma(s) + \epsilon(s). \label{eq:VstarApprox}
\end{equation}
Invoking the universal approximation theorem \cite[Thm.~1.5]{SCC.Sauvigny2012}, the function reconstruction error and its gradient are bounded as $\sup_{s \in \Omega}\|\epsilon(s)\|<\bar{\epsilon}$ and $\sup_{s \in \Omega}\|\nabla \epsilon(s)\| <\overline{\nabla\epsilon}$. The ideal value function weights satisfy $\|W\|<\bar{W}$, and the activation functions satisfy $\sup_{s \in \Omega}\|\sigma(s)\|<\bar{\sigma}$ and $\sup_{s \in \Omega}\|\nabla \sigma(s)\|<\overline{\nabla\sigma}$.

Since the ideal weight vector $W$ is unknown, both the value function and the feedback policy are approximated by parametrized functions.
Let $\hat{W}_{c}, \hat{W}_{a} \in \mathbb{R}^{L}$ denote adaptive weight estimates.  
The critic is defined by
\begin{equation}\label{eq:approxValue}
    \hat{V}(s, \hat{W}_{c}) \coloneqq \hat{W}_{c}^{\top}\sigma(s), \quad \forall s \in \Omega,
\end{equation}
and the actor is defined by
\begin{equation}\label{eq:approxControl}
    \hat{u}(s, \hat{W}_{a})
    \coloneqq
    -\frac{1}{2}\, R^{-1} G(s)^{\top}\nabla \sigma(s)^{\top}\hat{W}_{a}, \quad \forall s \in \Omega.
\end{equation}

\subsection{Bellman Error}
Using the critic $\hat{V}$ and actor $\hat{u}$, the temporal difference, also referred to as the Bellman error, is defined as
\begin{multline}\label{eq:BellmanError}
    \hat{\delta}(s, \hat{\theta}, \hat{W}_{a}, \hat{W}_{c})
    \coloneqq
    s^{\top} Q s
    +
    \hat{u}(s, \hat{W}_{a})^{\top} R \hat{u}(s, \hat{W}_{a}) \\  \nabla \hat{V}(s, \hat{W}_{c})^{\top}\bigl(A(s)\hat{\theta} + F(s) + G(s)\hat{u}(s, \hat{W}_{a})\bigr).
\end{multline}

Using the approximation representation \eqref{eq:approxValue}--\eqref{eq:approxControl}, we introduce the critic weight error and actor weight error as $\tilde{W}_{c} \coloneqq W - \hat{W}_{c}$ and $\tilde{W}_{a} \coloneqq W - \hat{W}_{a}$.
For the subsequent analysis, Bellman Error can be expressed in terms of the errors $\tilde{W}_{c}$ and $\tilde{W}_{a}$ as
\begin{equation}\label{eq:bellmanerrorfinal}
\hat{\delta} = -\omega^{\top}\tilde{W}_{c} 
    + \frac{1}{4}\tilde{W}_{a}^{\top} G_{\sigma}\tilde{W}_{a} 
    - W^{\top}\nabla\sigma\, A \tilde{\theta} 
    + \Delta, 
\end{equation}
where $\omega \coloneqq \nabla\sigma \left(A(s)\hat{\theta} + F(s) + G(s)\hat{u}(s, \hat{W}_{a})\right)$,
$G_{R} \coloneqq G R^{-1} G^{\top}$, $G_{\epsilon} \coloneqq \nabla\epsilon G_{R}\nabla\epsilon^{\top}$, $G_{\sigma} \coloneqq \nabla\sigma G R^{-1} G^{\top} \nabla\sigma^{\top}$, $\Delta \coloneqq \frac{1}{2}W^{\top}\nabla\sigma G_{R}\nabla\epsilon^{\top}+ \frac{1}{4}G_{\epsilon}- \nabla\epsilon F$, $A \coloneqq A(s)$, $F \coloneqq F(s)$, $\epsilon \coloneqq \epsilon(s)$, and $\sigma \coloneqq \sigma(s)$.

Online reinforcement-learning schemes commonly impose a persistence-of-excitation (PE) condition to ensure parameter convergence; however, such a condition cannot be enforced a priori and is generally not verifiable from online data. When a model of the system is available, Bellman-error (BE) extrapolation provides a virtual-excitation mechanism that yields a PE-like condition whose satisfaction can be assessed through a minimum-eigenvalue criterion. Since the BE can be evaluated at arbitrary states whenever the model is known, one may select a set of virtual states \B{$\{s_{k}(t)\}_{k=1}^{N}$} at each t and compute the BE at these extrapolated points. The extrapolated Bellman error is defined as
\begin{equation}
    \hat{\delta}_{k}
    \coloneqq \hat{\delta}(s_{k}, \hat{\theta}, \hat{W}_{a}, \hat{W}_{c}),
\end{equation}
and admits the representation 
\begin{equation}
\hat{\delta}_{k}
= -\omega_{k}^{\top}\tilde{W}_{c} 
    + \frac{1}{4}\tilde{W}_{a}^{\top} G_{\sigma_{k}} \tilde{W}_{a} 
    - W^{\top}\nabla\sigma_{k}\, A_{k}\tilde{\theta} 
    + \Delta_{k}, 
\end{equation}
where $\omega_{k} \coloneqq \nabla\sigma_{k} \left(A_{k}\hat{\theta} + F_{k} + G_{k}\hat{u}(s_{k},\hat{W}_{a}) \right)$, $G_{R_{k}} \coloneqq G_{k} R^{-1} G_{k}^{\top}$, $G_{\sigma_{k}} \! \coloneqq \! \nabla\sigma_{k} G R^{-1} G^{\top} \nabla\sigma_{k}^{\top}$, $G_{\epsilon_{k}}  \! \coloneqq \! \nabla\epsilon_{k} G_{R_{k}} \nabla\epsilon_{k}^{\top}$, 
 $\Delta_{k} \coloneqq \frac{1}{2}W^{\top}\nabla\sigma_{k} G_{R_{k}} \nabla\epsilon_{k}^{\top} + \frac{1}{4}G_{\epsilon_{k}} - \nabla\epsilon_{k} F_{k}$, $A_{k} \coloneqq A(s_{k})$, $F_{k} \coloneqq F(s_{k})$, $\epsilon_{k} \coloneqq \epsilon(s_{k})$, and $\sigma_{k} \coloneqq \sigma(s_{k})$. 
 
Under Assumption~\ref{ass:locallylipschitzfunctions} and the universal approximation property of neural networks, the Bellman-error residual satisfies $\sup_{s\in\Omega}\|\Delta(s)\| \le \overline{\Delta}$ for some constant $\overline{\Delta}>0$. Likewise, for any extrapolated state $s_{k}\in\Omega$, the corresponding residual obeys $\|\Delta_{k}\| \le \overline{\Delta}$.

\subsection{Update Laws for Actor and Critic Weights}
 
Based on the subsequent stability analysis, the critic update law is obtained by applying a gradient-type adaptation to the squared Bellman error. This yields
\begin{equation}\label{eq:WcUpdate}
    \dot{\hat{W}}_{c} = -k_{c_{1}}\Upsilon\frac{\omega}{\rho}\hat{\delta}-\frac{k_{c_{2}}}{N}\Upsilon \sum_{k=1}^{M} \frac{\omega_{k}}{\rho_{k}}\delta_{k},
\end{equation}
where $k_{c_{1}},k_{c_{2}}>0$ are adaptation gains and the normalization factors $\rho = 1 + \nu\,\omega^{\top}\omega$ and $\rho_{k} = 1 + \nu\,\omega_{k}^{\top}\omega_{k}$ use a regularization constant $\nu>0$.
The gain matrix $\Upsilon: \mathbb{R}_{\geq 0} \to \mathbb{R}^{L \times L}$ follows a standard recursive least-squares update,
\begin{equation}\label{eq:UpsilonUpdate}
    \dot{\Upsilon}
    = \beta_{c} \Upsilon -k_{c_{1}}\Upsilon\frac{\omega\omega^{\top}}{\rho^{2}}\Upsilon
    -\frac{k_{c}}{N}\Upsilon \sum_{i=1}^{M}
        \frac{\omega_{k}\omega_{k}^{\top}}{\rho_{k}^{2}} \Upsilon,
\end{equation}
with $\Upsilon(0)=\Upsilon_{0}$ and forgetting factor $\beta_{c}>0$. To guarantee parameter convergence, the subsequent stability analysis employs the following excitation condition.
\begin{assum}\label{ass:LearnCond}
There exists a constant $\underline{c}_{1}>0$ such that the trajectories $\{s_{k}(\cdot)\}_{k=1}^{N}$ satisfy
\begin{equation}\label{eq:learnCond}
    \underline{c}_{1}
    \le
    \inf_{t\ge 0}
    \lambda_{\min}\!
    \left(
        \frac{1}{N}
        \sum_{k=1}^{M}
        \frac{
            \omega_{k}(t)\,\omega_{k}^{\top}(t)
        }{
            \rho_{k}^{2}(t)
        }
    \right).
\end{equation}
\end{assum}
As shown in \cite{SCC.Kamalapurkar.Rosenfeld.ea2016} Lemma~1, provided Assumption~\ref{ass:LearnCond} holds and $\lambda_{\min}\{{\Upsilon_{0}^{-1}}\}> 0$, the update law in \eqref{eq:UpsilonUpdate} ensures that the least squares \TOC{gain} satisfies
\begin{equation}\label{eq:UpsilonBound}
	\underline{\Upsilon}I_{L}\preceq\Upsilon\left(t\right)\preceq\overline{\Upsilon}I_{L},   
\end{equation}
$\forall t\in\mathbb{R}_{\geq 0}$ and for some  $\overline{\Upsilon},\underline{\Upsilon}>0$.

The actor update law is obtained by gradient descent on a regularized loss formed from the same Bellman residuals. The resulting adaptation law is
\begin{multline}\label{eq:WaUpdate}
\dot{\hat{W}}_{a} = -k_{a_{1}}\left(\hat{W}_{a}-\hat{W}_{c}\right) -k_{a_{2}}\hat{W}_{a} + \frac{k_{c_{1}}G_{\sigma}^{\top}\hat{W}_{a}\omega^{\top}}{4\rho}\hat{W}_{c}  \\+ \sum_{k=1}^{M}\frac{k_{c}G_{\sigma_{k}}^{\top}\hat{W}_{a}\omega_{k}^{\top}}{4N\rho_{k}}\hat{W}_{c},
\end{multline}
where $k_{a_{1}}, k_{a_{2}} > 0$ are adaptation gains. 

The control input applied to the system \eqref{eq:augmentedSystem} uses actor weights generated by \eqref{eq:WcUpdate} as
\begin{equation}\label{eq:uControl}
    u(t) = \hat{u}\!\left(s(t),\,\hat{W}_{a}(t)\right),
    \qquad t\ge 0.
\end{equation}

The adaptive actor-critic update laws presented in this section will be utilized in the subsequent stability analysis to prove the uniform ultimate boundedness of the trajectories of the closed-loop system.


\section{Stability Analysis}
 To facilitate the analysis, let the concatenated vector 
 \begin{equation}
     Z \coloneqq \begin{bmatrix}
          s^{\top} & \tilde{z}^{\top} & \tilde{\theta}^{\top} & \tilde{W}_{c}^{\top} & \tilde{W}_{a}^{\top}
     \end{bmatrix}^{\top} \in \mathbb{R}^{n + p + 2L + 2}
 \end{equation}
represent the state of the closed-loop system and let a continuously differentiable candidate Lyapunov function, $V_{L}: \mathbb{R}^{n + p + 2L + 2} \times \mathbb{R}_{\geq 0} \to \mathbb{R}$, be defined as,
 \begin{multline}\label{eq:Lyap}
     V_{L}\left(Z, t\right) \coloneqq V^{*}(s) + \frac{1}{2}\tilde{z}^\top\tilde{z} + \frac{1}{2}\tilde{\theta}^\top \Gamma^{-1}(t)\tilde{\theta} \\ + \frac{1}{2}\tilde{W}_{c}^{\top}\Upsilon^{-1}(t)\tilde{W}_{c} + \frac{1}{2}\tilde{W}_{a}^{\top}\tilde{W}_{a}.
 \end{multline}
Using the bound in \eqref{eq:UpsilonBound} and since the candidate Lyapunov function \eqref{eq:Lyap} is positive definite, \cite[Lemma 4.3]{SCC.Khalil2002} can be used to conclude that it is bounded as
\begin{equation}
\underline{v}_{l}\left(\left\Vert Z\right\Vert \right)\leq V_{L}\left(Z,t\right)\leq\overline{v}_{l}\left(\left\Vert Z\right\Vert \right),\label{eq:VBound}
\end{equation}
for all $t \in \mathbb{R}_{\geq 0}$ and for all $Z\in\mathbb{R}^{n + p + 2L + 2}$, where $\underline{v}_{l},\overline{v}_{l}:\mathbb{R}_{\geq 0}\rightarrow\mathbb{R}_{\geq 0}$ are class $\mathcal{K}$ functions. Using the least square gain update law in \eqref{eq:UpsilonUpdate} and the bound in \eqref{eq:UpsilonBound}, the normalized regressor $\|\frac{\omega(t)}{\rho(t)}\|$ is bounded as $\|\frac{\omega(t)}{\rho(t)}\| \leq \frac{1}{2\sqrt{\nu\underline{\Upsilon}}}, \forall t \geq 0$. 
For any continous function $f:\Omega\to\mathbb{R}^m$ we denote by $\overline{f}$ any constant satisfying
$\sup_{s\in\Omega}\|f(s)\|\le \overline{f}$. When unambiguous we take $\overline{f}=\sup_{s\in\Omega}\|f(s)\|$. To facilitate the analysis, \B{define} the constants
\begin{align*}
\varpi_{1} &\coloneqq (k_{c_{1}}+k_{c_{2}})\,\overline{W}\,\overline{\nabla\sigma}\,\overline{A},\\
\varpi_{2} &\coloneqq (k_{c_{1}}+k_{c_{2}})\,\overline{\Delta},\\
\varpi_{3} &\coloneqq \tfrac{1}{2}\overline{W}\,\overline{G_{\sigma}}
    + \tfrac{1}{2}\overline{\nabla }\,\overline{G}\,\overline{\nabla \sigma}
    + k_{a_{2}}\,\overline{W}
    + \tfrac{1}{4}(k_{c_{1}}+k_{c_{2}})\,\overline{W}^{2}\,\overline{G_{\sigma}},\\
\varpi_{4} &\coloneqq k_{a_{1}}+\tfrac{1}{4}(k_{c_{1}}+k_{c_{2}})\,\overline{W}\,\overline{G_{\sigma}},\\
\varpi_{5} &\coloneqq \tfrac{1}{4}(k_{c_{1}}+k_{c_{2}})\,\overline{W}\,\overline{G_{\sigma}}.
\end{align*}
\TOC{Also, let the positive constant residual be defined as
$\iota \coloneqq \frac{\varpi_{2}^{2}}{2\varepsilon_2} + \frac{(\varpi_{3} + \varpi_{4})^{2}}{2\varepsilon_3} + \frac{1}{4}\overline{G}_{\epsilon}$}.
The following theorem establishes uniform ultimate boundedness of the trajectories of the closed-loop system.
\begin{theorem}\label{thm:theorem1}
   Let \TOC{$\overline{\mathbb{B}}_{\chi} \subset \Omega \times \mathbb{R}^{p + 2L + 1}$} be a closed ball of radius $\chi > 0$ containing the origin. Provided Assumptions \ref{ass:locallylipschitzfunctions}--\ref{ass:LearnCond} hold, if the unknown parameters $\tilde{\theta}$ are estimated using the adaptive update law in \eqref{eq:thetaUpdate} and \eqref{eq:gammaUpdate}, the sufficient conditions in \eqref{eq:suffconds} and    
\begin{equation}\label{eq:gainCond}
    \upsilon_l^{-1}\left(\iota\right) \leq  \overline{v}_{l}^{-1}\left(\underline{v}_{l}\left(\chi\right)\right),
\end{equation}
are satisfied, and the weights \B{$\hat{W}_{c}$, $\Upsilon$, $\hat{W}_{a}$, $\hat \theta$, and $\Gamma$  are updated according to \eqref{eq:WcUpdate}, \eqref{eq:UpsilonUpdate}, \eqref{eq:WaUpdate}, \eqref{eq:thetaUpdate}, and \eqref{eq:gammaUpdate} respectively}, then the concatenated state, $Z$, is locally uniformly ultimately bounded under the controller designed in \eqref{eq:uControl}. Furthermore, let $\mathcal{S}_0$ be chosen as
\begin{equation}\label{eq:S0}
\mathcal{S}_0 \coloneqq \{\, x(0)\in\mathcal{S} \mid Z(0) \leq \overline{v}_{l}^{-1}\!\big(v_{l}(\chi)\big) \,\},
\end{equation}
then every trajectory with $x(0)\in\mathcal{S}_0$, satisfies $x(t)\in\mathcal{S}$ for all $t\ge 0$.
\end{theorem}
\begin{proof}The orbital derivative of \eqref{eq:Lyap} along the trajectories of the closed loop system is given by
\begin{multline}\label{eq:VDot1}
    \dot{V}_{L}(Z, t) =  \nabla V(s)(A(s)\theta + F(s) + G(s)\hat{u}(s, \hat{W}_{a}))\\
   + \tilde{z}^\top \dot{\tilde{z}} 
    + \tilde{\theta}^\top \Gamma^{-1} \dot{\tilde{\theta}}
    - \frac{1}{2}\tilde{\theta}^\top \Gamma^{-1}\dot{\Gamma}\Gamma^{-1}\tilde{\theta}\\
    - \tilde{W}_{c}^{\top} \Gamma^{-1}\dot{\hat{W}}_{c} -\frac{1}{2} \tilde{W}_{c}^{\top} \Gamma^{-1}\dot{\Gamma}\Gamma^{-1}\tilde{W}_{c} -\tilde{W}_{a}^{\top}\dot{\hat{W}}_{a}.
\end{multline}
By substituting \eqref{eq:ZtildeDynamics}, \eqref{eq:tildeThetaDyn},  \eqref{eq:WcUpdate}, \eqref{eq:gammaUpdate} and \eqref{eq:WaUpdate}, then by the property of projection operators \cite[Lemma E.1. IV]{SCC.Krstic.Kanellakopoulos.ea1995}, using Cauchy-Schwarz inequality and completion of squares, \eqref{eq:VDot1} can be bounded as
\begin{multline}
 \dot{V}_{L}(Z, t) \leq  -\lambda_{\min}(Q)\|s\|^{2} - \gamma\|\tilde{z}\|^{2} - k_{\theta}\underline{\sigma}_{\theta}\|\tilde{\theta}\|^{2}
- k_{c_{2}}\underline{c}\|\tilde{W}_c\|^2 \\-(k_{a_1}+k_{a_2})\|\tilde{W}_a\|^2
+ \varpi_{1}\|\tilde W_c\|\,\|\tilde\theta\|
+ \varpi_{3}\|\tilde{W}_{a}\|
+ \varpi_{2}\|\tilde{W}_{c}\| \\ 
+ \varpi_{4}\|\tilde{W}_{a}\|
   + \tfrac{1}{4}\overline{G}_{\epsilon}
 + \varpi_{5}\|\tilde{W}_{a}\|^{2}.
\end{multline}
After applying Young's inequality, the bound on the derivative becomes
\begin{multline}\label{eq:Vdot_bound_intermediate}
\dot{V}_{L}(Z,t) \le
- \lambda_{\min}(Q)\|s\|^{2}
- \gamma\|\tilde{z}\|^{2}
- \Big(k_{\theta}\underline{\sigma}_{\theta}-\tfrac{\varpi_{1}^{2}}{2\varepsilon_1}\Big)\|\tilde{\theta}\|^{2} \\
- \Big(k_{c_{2}}\underline{c} - \tfrac{\varepsilon_1}{2}-\tfrac{\varepsilon_2}{2}\Big)\|\tilde{W}_c\|^{2}
- \Big((k_{a_1}+k_{a_2}) - \varpi_{5} - \tfrac{\varepsilon_3}{2}\Big)\|\tilde{W}_a\|^{2} \\
+ \tfrac{\varpi_{2}^{2}}{2\varepsilon_2}
+ \tfrac{(\varpi_{3} + \varpi_{4})^{2}}{2\varepsilon_3}
+ \tfrac{1}{4}\overline{G}_{\epsilon},
\end{multline}
where $\underline{c} \coloneqq\frac{\beta_{c}}{2\overline{\Upsilon} k_{c_{2}}} + \frac{\underline{c}_1}{2}$ is a positive constant and $\varepsilon_1,\varepsilon_2,\varepsilon_3>0$ are positive scalars. Let the effective quadratic coefficients be defined as
$\alpha_s \coloneqq \lambda_{\min}(Q)$,
$\alpha_z \coloneqq k_z$,
$\alpha_\theta \coloneqq k_{\theta}\underline{\sigma}_{\theta}-\frac{\varpi_{1}^{2}}{2\varepsilon_1}$,
$\alpha_{c} \coloneqq k_{c_{2}}\underline{c} - \frac{\varepsilon_1}{2}-\frac{\varepsilon_2}{2}$,
and $\alpha_{a} \coloneqq (k_{a_1}+k_{a_2}) - \varpi_{5} - \frac{\varepsilon_3}{2}$. 
Then \eqref{eq:Vdot_bound_intermediate} is compactly written as
\begin{multline}\label{eq:Vdot_final_bound}
\dot{V}_{L}(Z,t) \le
- \alpha_s\|s\|^2 - \alpha_z\|\tilde{z}\|^2 - \alpha_\theta\|\tilde\theta\|^2
- \alpha_c\|\tilde W_c\|^2\\ - \alpha_a\|\tilde W_a\|^2 + \iota.
\end{multline}
To make the quadratic part strictly negative definite, choose $\varepsilon_1,\varepsilon_2,\varepsilon_3>0$ and gains such that
\begin{subequations}\label{eq:suffconds}
\begin{align}
\alpha_\theta &= (k_{\theta} - 1)\underline{\sigma}_{\theta} + \frac{\beta_{\theta}}{k_{\theta}\overline{\Gamma}} -\frac{\varpi_{1}^{2}}{2\varepsilon_1} > 0, \label{eq:cond_theta}\\
\alpha_c &= k_{c_{2}}\underline{c} - \frac{\varepsilon_1}{2}-\frac{\varepsilon_2}{2} > 0, \label{eq:cond_c}\\
\alpha_a &= (k_{a_1}+k_{a_2}) - \varpi_{5} - \frac{\varepsilon_3}{2} > 0. \label{eq:cond_a}
\end{align}
\end{subequations}

If the gain values $k_\theta,k_{c_{2}},k_{a_1},k_{a_2}$ are chosen so that the right-hand sides above are positive, then $\alpha_s,\alpha_z,\alpha_\theta,\alpha_c,\alpha_a$ are all positive. Under these choices we obtain from \eqref{eq:Vdot_final_bound}
\begin{equation}
    \dot{V}_{L} \le -\upsilon_{l}\left(\|Z\|\right), \quad \forall \|Z\| \geq \upsilon_{l}^{-1}(\iota),
\end{equation}
for all $t \in \mathbb{R}_{\geq 0}$, and for all $Z \in \overline{\mathbb{B}}_{\chi}$,
where $\upsilon_{l}: \mathbb{R}_{\geq 0} \to \mathbb{R}_{\geq 0}$ is a class $\mathcal{K}$ function that satisfies
\begin{equation}
   \upsilon_{l}\left(\|Z\|\right) \leq \frac{1}{2}\min\{\alpha_s,\alpha_z,\alpha_\theta,\alpha_c,\alpha_a\}\|Z\|^{2}.  
\end{equation}
Consequently, if the sufficient conditions in \eqref{eq:suffconds} are met and the bound in \eqref{eq:VBound} is maintained, then Theorem~4.18 of \cite{SCC.Khalil2002} guarantees a local uniform ultimate boundedness property for the closed-loop signal $Z$. In particular, every solution satisfying the initial bound $\|Z(0)\| \le \overline{v}_{l}^{-1}\!\big(v_{l}(\chi)\big)$ obeys $\limsup_{t\to\infty} \|Z(t)\| \le \underline{v}_{l}^{-1}\!\Big(\overline{v}_{l}\!\big(v_{l}^{-1}(\iota)\big)\Big)$. Moreover, the concatenated state trajectories remain confined to the ball $\overline{\mathbb{B}}_{\chi}$ for all $t \in \mathbb{R}_{\ge 0}$. Since the actor weights \TOC{estimation errors $\tilde{W}_{a}$ are uniformly ultimately bounded}, the resulting control policy $\hat{u}$ constitutes an approximation of the optimal policy $u^{*}$. 

Hence, for every initial condition $x(0)\in \mathcal{S}_0$, the true barrier state $z(t)$ remains bounded for all $t\ge 0$. Moreover, since boundedness of $z(t)$ is equivalent to $x(t)\in \mathcal{S}$ by Lemma~\ref{lem:cone}, $x(t)$ remains in the safe set $\mathcal{S}$ for all $t\ge 0$, thereby ensuring safety.
\end{proof}

\begin{theorem}\label{thm:theorem2}
Provided the Assumptions~\ref{ass:locallylipschitzfunctions} - \ref{ass:fullRank} hold, if the unknown parameters and the state estimates are updated using \eqref{eq:ZhatDynamics} and \eqref{eq:thetaUpdate}, and the gains are selected such that
\begin{equation}\label{eq:finalCond}
\frac{\beta_{\theta}}{k_{\theta}\overline{\Gamma}} >
\underline{\sigma}_{\theta},
\end{equation}
is satisfied, then the error systems in \eqref{eq:ZtildeDynamics} and \eqref{eq:tildeThetaDyn} are exponentially stable.
\end{theorem}
\begin{proof}
Let a continuously differentiable candidate Lyapunov function, $\tilde{V}: \mathbb{R} \times \mathbb{R}^{p} \times \mathbb{R}_{\geq 0} \to \mathbb{R}$, be defined as,
\begin{equation}
\label{eq:lyapunovfuncFromObserverStability}
    \tilde{V}(\tilde{z}, \tilde{\theta}, t) = \frac{1}{2}\tilde{z}^\top \tilde{z} + \frac{1}{2}\tilde{\theta}^\top\Gamma(t)^{-1}\tilde{\theta}.
\end{equation}
There exist constants $\underline{\tilde{v}},\overline{\tilde{v}}>0$ such that
$\underline{\tilde{v}}\|( \tilde z,\tilde\theta)\|^2 \le \tilde V(\tilde z,\tilde\theta,t) \le \overline{\tilde{v}}\|( \tilde z,\tilde\theta)\|^2
$, $\forall t\ge 0$, which holds locally by positive definiteness of \(\tilde V\) and continuity of \(\Gamma(t)\).
By the property of projection operators in \cite[Lemma E.1. IV]{SCC.Krstic.Kanellakopoulos.ea1995}, the orbital derivative of $\tilde{V}$ along the 
trajectories \eqref{eq:ZtildeDynamics}, \eqref{eq:tildeThetaDyn}, and \eqref{eq:gammaUpdate} is bounded as
\begin{multline}\label{eq:Vdotestimator}
    \dot{\tilde{V}}(\tilde{z}, \tilde{\theta}, t) \leq \tilde{z}^\top \left(\Phi(z+\beta_0)\nabla h(x)Y(x)\tilde{\theta}-\gamma\tilde{z} \right) \\ 
    - k_{\theta}\tilde{\theta}^\top\Sigma_{\mathcal{Y}}\tilde{\theta} - \tilde{\theta}^\top Y(x)^{\top}\nabla  h(x)^{\top}\Phi(z+\beta_0)^{\top} \tilde{z}
    \\-\frac{1}{2}\tilde{\theta}^\top \Gamma^{-1}
    \left(\beta_{\theta}\Gamma -k_\theta \Gamma 
    \Sigma_{\mathcal{Y}} \Gamma\right)
    \Gamma^{-1}\tilde{\theta}
\end{multline}
when $\dot\Gamma = \beta_{\theta}\Gamma -k_\theta \Gamma 
\Sigma_{\mathcal{Y}} \Gamma $ in \eqref{eq:gammaUpdate}. When $\dot\Gamma = 0$ in \eqref{eq:gammaUpdate}, then the Lie derivative of $\tilde{V}$ is bounded as
\begin{multline}\label{eq:Vdotestimator2}
\dot{\tilde{V}}(\tilde{z}, \tilde{\theta}, t) \leq \tilde{z}^\top \left(\Phi(z+\beta_0)\nabla h(x)Y(x)\tilde{\theta}-\gamma\tilde{z} \right) \\ 
 - k_{\theta}\tilde{\theta}^\top\Sigma_{\mathcal{Y}}\tilde{\theta} - \tilde{\theta}^\top Y(x)^{\top}\nabla  h(x)^{\top}\Phi(z+\beta_0)^{\top} \tilde{z}.
\end{multline}
Using Assumption~\ref{ass:fullRank}, applying the triangle inequality, and provided the gain condition in \eqref{eq:finalCond} is satisfied, the orbital derivative \eqref{eq:Vdotestimator} can be bounded, \TOC{for any update of $\dot\Gamma$ in \eqref{eq:gammaUpdate}, as}
\begin{equation}\label{eq:tildeVdot}
\tilde{V}(\tilde{z}, \tilde{\theta}, t) 
\leq 
-\,\gamma\,\|\tilde z\|^{2}
- k_{\theta}\underline\sigma_{\theta}\,
\|\tilde\theta\|^{2}.
\end{equation}
Hence, by Theorem 4.8 of \cite{SCC.Khalil2002}, the error dynamics \eqref{eq:ZtildeDynamics} and \eqref{eq:tildeThetaDyn} are exponentially stable. \TOC{Consequently, under the ICL-based update law \eqref{eq:thetaUpdate}, the parameter estimates $\hat{\theta}(t)$ converge exponentially to the true parameters $\theta$, i.e.,
$\lim_{t \to \infty} \tilde{\theta}(t) = \lim_{t \to \infty} \big(\hat{\theta}(t) - \theta\big) = 0$}.
\end{proof}
\section{Simulation Study}
This section will examine the effectiveness of the control policy proposed in \eqref{eq:approxControl} \B{using} an obstacle avoidance problem. The system dynamics for this study is a nonlinear control-affine system of the form in \eqref{eq:dynamics} with state $x = [x_{1}, x_{2}]^\top$, where
\begin{equation}\label{eq:simDyn2} Y(x) = \begin{bmatrix} x_1 & x_2 & 0 & 0 &\\ 0 & 0& x_1 +x_2 & x_1^2x_2 \end{bmatrix}, \end{equation} $g(x) = [0, \; \cos(2x_{1}) + 2]^\top$, $f(x) = [\ 0 \quad  0\ ]^\top$, and $\theta= \begin{bmatrix}
    \theta_1 & \theta_2 & \theta_3 & \theta_4
\end{bmatrix}^\top = \begin{bmatrix}
    -1 & -1 & -0.5 & -0.5
\end{bmatrix}^\top$.

For the simulation, the safe set is defined by \eqref{eq:safeSet1} where the constraint is given by $h(x)= (x_1-1)^2+ (x_2-2)^2 -0.5^2$. The objective of the control policy is to ensure that the agent avoids the obstacle and therefore remains within the safe set specified by Definition~\ref{defn:safety}.  The penalty or cost for states and control effort in \eqref{eq:costFunctional} are respectively chosen as $Q=I_3$ and $R=1$. The estimates for $ \theta$ denoted by $\hat \theta= \begin{bmatrix}
\hat \theta_1 & \hat \theta_2 & \hat \theta_3 & \hat \theta_4
\end{bmatrix}^\top$. \TOC{The} initial condition for the system in \eqref{eq:simDyn2} is $ x(0) = [2.5, \; 4]^\top $. The initial values \B{selected} for the estimates, weights, and gains are $\hat{z}(0)=0$, $\hat\theta(0)=0_{4\times1}$, $\hat{W_c}=0.5{1}_{6 \times 1}$, $ \hat{W_a}=0.5{1}_{6 \times 1}$, $\Gamma(0)= 10 I_{4}$, \TOC{and} $\Upsilon(0)= 0.01 I_{6}$ . The barrier function in \eqref{eq:barrierFunction} \TOC{is} selected as $\beta(x)=B(h(x))=  \frac{K}{h(x)},\quad \forall \ x \in \ \mathbb{R}^2$ so that $\Phi(\beta)=-\frac{K}{\beta^2}$. The observer uses $\gamma = 3$, and the barrier gain is selected as $K = 0.01$. 

The control and learning gains used in the simulation are as follows. For the ADP framework, the gains are set to $\nu = 2$, $k_{c1} = 1$, $k_{c2} = 1$, $k_{a1} = 2$, $k_{a2} = 1$, and $\beta_c = 0.1$, with bounds $\underline{\Upsilon} = 0$ and $\overline{\Upsilon} = 1000$. The ICL gains are chosen as $k_{\theta} = 50$, $\kappa = 1$, and $\beta_{\theta} = 1$. The simulation uses 100 fixed Bellman error extrapolation points $s_k$ placed within a $4 \times 4$ square centered at the origin of the $s=[x_1 \quad x_2 \quad z]^\top$ coordinate system. The basis for the value function is \TOC{selected} as $\sigma(s)=\begin{bmatrix}s_1^2 & s_2^2 & s_3^2 & s_1s_2 & s_2s_3 & s_3s_1 \end{bmatrix}^\top$. The \TOC{s}imulation study will examine the effectiveness of the control policies for the developed BaS-RL-ICL framework, \B{the CBF-RL-ICL framework from \cite{SCC.Cohen.Serlin.ea2023}}, and \TOC{the RL-ICL} without any safety constraints, denoted by Cases 1, 2, and 3, respectively.
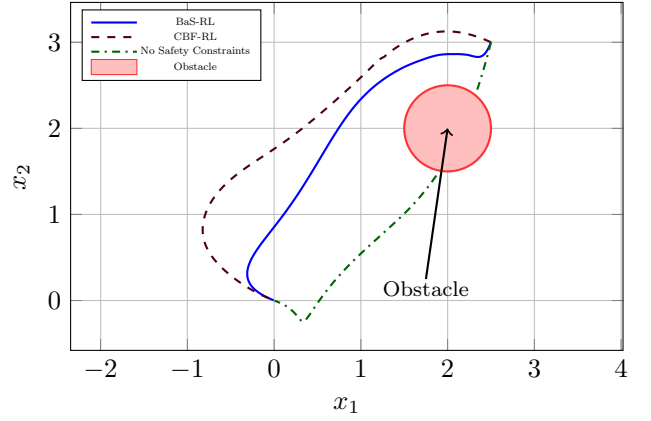
\begin{figure}
    \centering
    \begin{tikzpicture}
  \begin{axis}[
    width=\linewidth,
    height=0.7\linewidth,
    xlabel={$x_1$},
    ylabel={$x_2$},
    grid=both,
    axis equal,
    legend style={
      at={(0.02,0.98)},
      anchor=north west,
      nodes={scale=0.5, transform shape}
    }
  ]

    \draw[
      fill=red!25,
      draw=red!80,
      thick,
    ] (axis cs:2,2) circle [radius=0.5];

    \addlegendimage{area legend, fill=red!25, draw=red!80}
    \addlegendentry{\small Obstacle}

    \draw[->, thick]
      (axis cs:1.75, 0.25)
      node[above, yshift=-10pt]{\small Obstacle}
      -- (axis cs:2, 2);

    \addplot+[
      thick,
      blue,
      solid,
      mark=none,
    ] table[
      x index=1,
      y index=2,
      col sep=space,
    ]{Data/Case-01/x.dat};
    \addlegendentry{\small BaS-RL}

    \addplot+[
      thick,
      purple!40!black,
      dashed,
      mark=none,
    ] table[
      x index=1,
      y index=2,
      col sep=space,
    ]{Data/Case-02/x.dat};
    \addlegendentry{\small CBF-RL}

    \addplot+[
      thick,
      green!40!black,
      dash dot,
      mark=none,
    ] table[
      x index=1,
      y index=2,
      col sep=space,
    ]{Data/Case-03/x.dat};
    \addlegendentry{\small No Safety Constraints}

  \end{axis}
\end{tikzpicture}
    \caption{Phase-space trajectories of the system}  
    \label{fig:trajectory}
\end{figure}

\begin{figure}
    \centering
    \begin{tikzpicture}
  \begin{axis}[
    width=\linewidth,
    height=0.45\linewidth,
    xlabel={$t$ (in sec)},
    ylabel={$\tilde\theta$},
    grid=both,
    legend style={nodes={scale=0.7, transform shape}},
    xmin = 0,
    enlarge y limits=0.01,
    enlarge x limits=0,
  ]


  \addplot+[thick, red!70!black, 
      mark=triangle,
      mark size=1.5pt,
      mark repeat=60,
      mark options={solid}]
    table[x index=0, y index=1]{Data/Case-01/theta_hat.dat};

  \addplot+[thick, blue, 
      mark= diamond,
      mark size=1.5pt,
      mark repeat=60,
      mark options={solid}]
    table[x index=0, y index=2]{Data/Case-01/theta_hat.dat};

  \addplot+[thick, purple!80!black,
      mark=*,
      mark size=1.5pt,
      mark repeat=60,
      mark options={solid}]
    table[x index=0, y index=3]{Data/Case-01/theta_hat.dat};

  \addplot+[thick, green!50!black, 
      mark=square,
      mark size=1.5pt,
      mark repeat=60,
      mark options={solid}]
    table[x index=0, y index=4]{Data/Case-01/theta_hat.dat};

  \addlegendimage{thick, red!70!black}
  \addlegendentry{$\tilde\theta_1$}

  \addlegendimage{thick, blue}
  \addlegendentry{$\tilde\theta_2$}

  \addlegendimage{thick, purple!80!black}
  \addlegendentry{$\tilde\theta_3$}

  \addlegendimage{thick, green!50!black}
  \addlegendentry{$\tilde\theta_4$}

  \end{axis}
\end{tikzpicture}
    \caption{{Evolution of the parameter estimate error with time. }}  
   \label{fig:theta_hat}
\end{figure}
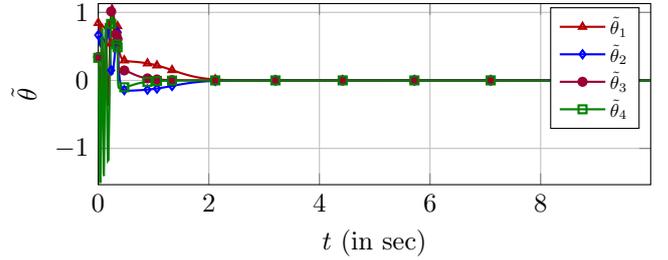

\subsection{Discussion}
The effectiveness of the framework is compared with the CBF-RL technique described in \cite{SCC.Cohen.Serlin.ea2023}, \B{Figure \ref{fig:trajectory} shows the effectiveness of BaS-augmentation-based control policy in avoiding the obstacle. Also, the CBF-based approach maintained larger safety margins in comparison to the BaS-based control policy}. Figure \ref{fig:theta_hat} demonstrates that the system parameter error converged exponentially to zero, consistent with \B{Theorem~\ref{thm:theorem2}}. Together, these results highlight the improved efficiency, stability, and learning performance of the BaS framework in comparison \TOC{with} existing CBF-based techniques.

\section{Conclusion}
In this paper, we introduced a BaS-based safe control framework for nonlinear control-affine systems with unknown parameters. The results show that the method preserves safety without prior parameter knowledge and drives the estimates towards the true values while maintaining stable system behavior. The framework has some limitations: the policy is near-optimal rather than asymptotically optimal, and finite excitation is required for full convergence. Furthermore, augmenting the system dynamics with BaS dynamics increases the computational complexity of the optimal control problem due to the increased state dimension and model nonlinearity. Despite these limitations, as shown in the simulation results, the developed ICL-BaS-RL framework guarantees near-optimal safety and stability. Future work will extend the framework to more complex environments, including settings with time-varying obstacles and dynamic disturbances, and investigate conditions that provide global safety guarantees under less restrictive assumptions.

{\small
\bibliography{scc,sccmaster,scctemp,ifacconf}                                                    
}


\end{document}